\documentclass[12pt,a4paper]{article}
\usepackage[utf8]{inputenc}
\usepackage[T1]{fontenc}
\usepackage[english]{babel}
\usepackage{amsmath}
\usepackage{amsfonts}
\usepackage{amssymb}
\usepackage{graphicx}
\usepackage{amsthm}
\usepackage{hyperref}
\usepackage[open,openlevel=1]{bookmark}
\newtheorem{theorem}{Theorem}[section]
\newtheorem{lemma}{Lemma}[section]
\newtheorem{corollary}{Corollary}[section]

\newtheorem{definition}{Definition}

\newtheorem{question}{Question}
\newtheorem*{acknowledgement}{Acknowledgement}
\title{Some Characterisations of $ p $-adic Analytic Groups}

\author{Chaitanya Ambi\footnote{Chennai Mathematical Institute, H1, SIPCOT IT Park, Siruseri, Kelambakkam 603103, India. \newline Email: chaitanya.ambi@gmail.com }}
\begin{document}
 \maketitle
\begin{abstract}

        We give three necessary and sufficient conditions for a pro-$ p$ group to be $ p $-adic analytic. We show that a noetherian pro-$p $ group having finite \textit{chain length} has a finite rank and conversely. We further deduce that a noetherian pro-$p$ group has a finite rank precisely when it satisfies the \textit{weak} descending chain condition. Using these results, we resolve a conjecture posed by Lubotzky and Mann in the affirmative within the class of noetherian groups which are \textit{countably based}.\\

        Using these results, we answer a related conjecture about pro-$ p $ groups for the case of countably based pro-$ p $ groups. Namely, we prove that if every closed but non-open subgroup of a countably based pro-$ p $ group has finite rank, then the group is $ p $-adic analytic and conversely.
\end{abstract}
\textbf{Keywords}: Pro -$ p $ Groups, $ p $-adic Analytic Groups.\\
\textbf{MSC[2010]}: 20E18, 20E34.
\section{Introduction.}\label{intro}
\indent

    For a prime $ p $, consider a pro-$p$ group $G$. If there is an integer $r\geq 0$ such that every closed subgroup of $G$ can be generated topologically by at most $r$ elements, $G$ is said to have \textit{finite rank}. In this case, the rank of $G$, denoted by $rk(G)$, is the minimal such integer.\\

    A celebrated result by Lazard asserts that pro-$ p $ groups of finite rank are precisely the $ p $-adic analytic groups. This class of pro-$ p $ groups is, therefore, the $ p $-adic analogue of real Lie groups. Several interesting characterisations of $ p $-adic analytic groups are known (see \cite{dixon}, Interlude A, p. 97). Lubotzky and Mann \cite{lubotzky} posed the following problem which aims at yet another characterisation of $ p $-adic analytic groups.
\begin{question}\label{conj-noe}
        Does every noetherian pro-$ p $ group have finite rank?
\end{question}

Here, we call a pro-$ p $ group satisfying the ascending chain condition on its closed subgroups \textit{noetherian}. This property is equivalent to each closed subgroup being finitely generated. (By convention, we call a pro-$p  $ group \textit{finitely generated} if it contains a finitely generated dense subgroup.)

\begin{definition}\label{chain_length}
    Let $ G $ be an infinite pro-$ p $ group. A chain of closed subgroups in $ G $
    \begin{equation*}
        1=G_{0} \lneq G_{1}\lneq_{c} G_{2}\lneq_{c} \dots \quad \lneq_{c}G_{m}=G
    \end{equation*} such that $ [G_{i+1}: G_{i}] =\infty $ for each $ 0\leq i \leq m-1 $ is said to be \textit{proper}. The \textit{chain length} $ l(G)\in \mathbb{N}\cup \{\infty \} $ of $ G $ is the supremum of $ m $ over all proper chains.
\end{definition}

        As already observed in \cite{klopsch} (see the paragraph following Thm. 5.11, p. 34), we have $ l(G)\leq \dim (G) $ whenever $ G $ is $ p $-adic analytic. The first result in this article shows that the answer to Q. \ref{conj-noe} is in the affirmative whenever a noetherian pro-$p$ group has finite chain length.
\begin{theorem}\label{thm_noe}
    A noetherian pro-$ p $ group $ G $ satisfies $ rk(G)<\infty $ precisely when $l(G)<\infty$.
\end{theorem}
We define an analogue of the artinian condition in order to quote the next result.
\begin{definition}\label{weak_dcc}
 A pro-$p$ group $G$ is said to satisfy the \textit{weak} descending chain condition if every proper descending chain of closed subgroups
 \begin{equation*}
    1\lneq_{c} \dots \quad \lneq_{c}G_{1}\lneq_{c}G_{0}=G
 \end{equation*}
    (with $[G_{i+1}:G_{i}]=\infty$ for all $i\geq 0$) terminates.
\end{definition}
    The following result shows how the above property is related to rank.
\begin{theorem}\label{thm_art}
    A noetherian pro-$p$ group $G$ has a finite rank precisely when it satisfies the weak descending chain condition.
\end{theorem}

 This result, together with \ref{thm_noe} provides a partial answer to Q. \ref{conj-noe}. To obtain a stronger result, we consider another conjecture aiming at a new characterisation of $p$-adic analytic groups (given in \cite{horizons}, Problem 2, p. 411).
\begin{question}\label{conj-non-open}
    Let $ G $ be a pro-$ p $ group such that each of its closed but non-open subgroup has finite rank. Must $ G $ have finite rank?
\end{question}

        Surprisingly, the above conjecture turns out to be related to Q. \ref{conj-noe}. Specifically, we consider Q. \ref{conj-non-open} when $ G $ is countably based. Here, we call a topological group \textit{countably based} if the identity element (and hence any other element) has a countable neighbourhood basis. Equivalently, the group has only countably many finite continuous images. We prove the following.
\begin{theorem}\label{thm_non-open}
    Let $ G $ be a countably based pro-$ p $ group. If each of its closed but non-open subgroups has finite rank, then $ rk(G)<\infty $ and conversely.
\end{theorem}
We further note that by Lemma \ref{not_f.g.}, a possible counter-example to Q. \ref{conj-non-open} could not be finitely generated or even countably based. Furthermore, Thm. \ref{thm_non-open} allows us to deduce the following.
\begin{corollary}\label{cor_final}
 Let $G$ be a noetherian pro-$p$ group. Then $G$ is countably based if and only if $rk(G)<\infty$.
\end{corollary}

    This answers Q. \ref{conj-noe} within the class of countably based noetherian groups.
\section{Preliminary Results.}\label{prelim}
\indent

For a prime $ p $, we shall denote the ring of $ p $-adic integers by $ \mathbb{Z}_{p} $. If $ G $ is a pro-$ p $ group and $ g\in G $, then the $ \mathbb{Z}_{p}$-powered closed subgroup generated by $ g $ will be denoted by $ g^{\mathbb{Z}_{p}} $. The relation $ H \leq K$ will indicate that $ H $ is a subgroup of $ K $. If this containment is proper, we shall write $H\lneq K$. \\

Furthermore, the suffixes $ 'c' $ and $'o' $ to these relations will stand for 'closed subgroup' and 'open subgroup', respectively. For instance, if $H\lneq_{c}K$ and $H$ is not open, then we must have $[K:H]=\infty$ (see \cite{dixon}, Prop. 1.2(i), p. 15). Lastly, for a subset $ X\subseteq G $, we shall write $ \langle X \rangle $ for the \textit{closed} subgroup of $ G $ generated by $ X $.\\

We call a pro-$ p $ group $ G $ \textit{countably generated (resp., finitely generated) as a} $ \mathbb{Z}_{p} $\textit{-powered group} if there exists a countable (resp., finite) subset $ X\subset G  $ such that each element $ g\in G $ equals the \textit{finite} product
\begin{equation*}
    g= \prod_{1\leq i \leq n}x_{i}^{\lambda_{i}}, \quad \lambda_{i} \in \mathbb{Z}_{p}
\end{equation*}
for some $ \{x_{1}, x_{2}, \dots x_{n}\} \subseteq X $. We quote a result here; see \cite{dixon} (Thm. 3.17, p. 53) for a proof.
\begin{lemma}\label{dix}
    A pro-$ p $ group $ G $ has finite rank if and only if $ G $ is countably generated as a $ \mathbb{Z}_{p} $-powered group. In this case, $ G $ is the product of finitely many of its procyclic subgroups and hence is finitely generated as a $ \mathbb{Z}_{p} $-powered group.
\end{lemma}
\indent

In the following lemma, we call a pro-$ p $ group \textit{virtually procyclic} if it contains a procyclic subgroup of finite index.
\begin{lemma}\label{l_one}
    A finitely generated pro-$ p $ group $ G $ with $l(G)=1 $ is virtually procyclic.
\end{lemma}
\begin{proof}
\indent

Each closed subgroup of the form $ g^{\mathbb{Z}_{p}}, g\in G $ is procyclic (see \cite{dixon}, Prop. 1.28, p. 30). If we had $|g^{\mathbb{Z}_{p}}|<\infty $ for all $ g\in G $, then $ G $ would be torsion. As $ G $ is finitely generated, it would follow from a result by Zelmanov (see \cite{horizons}, Chap. 1, Cor. 2.4, p. 4) that $ |G| < \infty$. This is in contradiction with our hypothesis that $l(G)=1 $.\\

Thus, $ G $ contains an infinite subgroup $ x^{\mathbb{Z}_{p}} \leq_{c} G $ for some $ x\in G $. Since $ l(G)=1 $, we must have $ [G: x^{\mathbb{Z}_{p}}] < \infty$. By \cite{dixon} (see Prop. 1.28(c), p. 30), $ G $ is virtually procyclic.
\end{proof}
\begin{lemma}\label{not_f.g.}
    Let $ G $ be a countably based pro-$ p $ group which is not finitely generated. There exists a closed but non-open subgroup $ H\leq_{c} G $ such that $ rk(H)=\infty $.
\end{lemma}
\begin{proof}
          By \cite{dixon} (Prop. 1.14, p. 24), the Frattini subgroup $ \Phi(G) \unlhd_{c} G$ is closed but not open. Consider the countably based, abelian pro-$ p $ group $ G/\Phi(G) $. Such groups are classified in \cite{kiehlmann} (see Thm. 1.4). Accordingly, $ G/\Phi(G) $ is $ p $-elementary abelian with a countably infinite number of generators. Split the set of its generators into two disjoint infinite subsets. By omitting one of these subsets, we get an infinite proper subgroup $ K \lneq_{c} G/\Phi(G)$. It corresponds to a closed subgroup $ H:=K\Phi(G)\leq_{c} G $. As the index $ [G/\Phi(G) : K]  $ is infinite, $ K $ cannot be open in $ G/\Phi(G) $. Hence, $ H$ is not open in $ G $ either. Finally, $ rk(H) $ is also infinite as $ H $ is not finitely generated.
\end{proof}
\indent

    The following Lemma is based on Zelmanov's solution of the \textit{Resticted Burnside Problem}.
\begin{lemma}\label{RBP}
    Let $G$ be a noetherian pro-$p$ group. Let $H\unlhd_{c}G$ be a closed normal subgroup such that $[G:H]=\infty$. Then there exists an element $x\in G$ such that  $ x^{k} \not \in H $ for each $k \in \mathbb{Z}\backslash \{0\}$.
\end{lemma}
\begin{proof}
\indent

                If a non-zero integral power of each $y\in G$ were to lie in $H$, the quotient $G/H$ will be a torsion pro-$p$ group. Since $G$ is noetherian, $G/H$ is a finitely generated as well. The finiteness of $G/H$ follows from a result by Zelmanov \cite{horizons}(see Chap. 1, Cor. 2.4, p. 4). Hence, $[G:H] <\infty$, which is contrary to the hypothesis of the Lemma.
    \end{proof}
\indent

    The family of closed subgroups having infinite index in a given \textit{noetherian} pro-$p$ group contains a maximal closed, non-open subgroup. This observation is the basis of the following Lemma.
\begin{lemma}\label{lem_max}
Let $G$ be an infinite noetherian pro-$p$ group. Let $H$ be a closed subgroup of $G$ which is maximal among the closed subgroups of infinite index. If $rk(H)<\infty$, then $rk(G)<\infty$.
\end{lemma}
\begin{proof}

    Since $G$ itself is finitely generated, we have $\Phi(G)\unlhd_{o}G$ for the Frattini subgroup of $G$. Hence, the subgroup $N\leq_{c}H$ defined as
    \begin{equation*}
     N:=\Phi(G)\cap H
    \end{equation*}
    is open in $H$ (by \cite{dixon}, Pro. 1.2(v), p. 16). As $H$ is finitely generated, we must have $[H: N]<\infty$. Note also that by \cite{clement}(see Chap. 7.5, Lemma 7.12, p. 293), we have $N\unlhd_{c}G$ with the isomorphism $\Phi(G/N)\approx \Phi(G)/N$. Thus, in fact,
    \begin{equation*}
     N\unlhd_{c}\Phi(G).
    \end{equation*}
    We also have
    \begin{equation*}
     \dfrac{\Phi(G)H}{H} \approx \dfrac{\Phi(G)}{N} \approx \Phi(G/N).
    \end{equation*}
    In the above equation, $\Phi(G)H$ is, \textit{a fortiori}, a \textit{group}. Observe that both $[G:\Phi(G)]$ and $[H:N]$ are finite. But $\Phi(G)H$ is open in $G$ (and hence, $[G:\Phi(G)H]<\infty$). As $[G:H]$ is infinite, so is the group $\Phi(G)H/H$.\\

    But now, $H$ persists to be maximal in $\Phi(G)H$ with respect to the property of having infinite index and a finite rank. The noetherian pro-$p$ group $\Phi(G)H/H$ cannot contain any \textit{proper} closed subgroups of infinite index and finite rank (otherwise, $H$ will cease to be maximal in $G$).\\

    We use Lemma \ref{RBP} to obtain $x\in \Phi(G)H$ such that none of the non-zero integral powers of $ x$ lies in $H$. Setting $\bar{x}:=xH\in \Phi(G)H/H$, the argument in the preceeding paragraph shows that the closed procyclic subgroup $\bar{x}^{\mathbb{Z}_{p}}\leq_{c}\Phi(G)H/H$ cannot have infinite index. Hence,
    \begin{equation*}
              [ \Phi(G)H/H: \bar{x}^{\mathbb{Z}_{p}}]< \infty,
             \end{equation*}
    which shows that $l(\Phi(G)H/H)=1$. By Lemma \ref{l_one}, $rk(\Phi(G)H/H)=1$. Thus,
    \begin{equation*}
    rk(\Phi(G)H)\leq rk(H)+rk(\Phi(G)H/H)<\infty.
    \end{equation*}
    This completes the proof once we observe that $[G:\Phi(G)H]<\infty$.
\end{proof}
\section{Proof of Theorem \ref{thm_noe}.}
\begin{proof}
\indent

        If $rk(G)<\infty$, then $G$ is $p$-adic analytic and $l(G)\leq \dim(G)<\infty$. For the converse, we shall proceed by induction on the chain length $ l(G) $. The case when $ l(G)=1 $ has already been addressed by Lemma \ref{l_one} as all virtually procyclic groups have rank one.\\

        Next, let $ 2\leq l(G) <\infty. $ Define
        \begin{equation*}
                \mathcal{G}:=\lbrace H: H\lneq_{c} G , \quad l(H)=l(G)-1\rbrace.
               \end{equation*}
               By the definition of chain length, it is clear that $\mathcal{G}\not = \emptyset$. Since $G$ is noetherian as well, $\mathcal{G}$ contains a maximal closed subgroup $H_{max}\lneq_{c}G$ having infinite index. But the induction hypothesis is that every closed pro-$ p $ subgroup $ K \lneq_{c} G $ having
               $ 1\leq l(K)\leq l(G)-1 $ has finite rank. This applies, in particular, to $H_{max}$. This completes the proof in view of Lemma \ref{lem_max} applied to $G$.
\end{proof}
\section{Proof of Theorem \ref{thm_art}.}
\begin{proof}
\indent

If $rk(G)<\infty$, then no proper descending chain in $G$ can have a chain length exceeding $\dim(G)$. We now prove the converse implication.\\

    Assume that $G$ is an infinite noetherian group satisfying the weak descending chain condition. In other words, no ascending or descending \textit{proper} chain in $G$ can be infinite. We introduce a partial order on all finite proper chains of closed subgroup of $G$ by refinement. More precisely, given two finite proper chains $\mathcal{A},\mathcal{B}$,
\begin{align*}
 \mathcal{A}:\quad 1\lneq_{c} \dots \quad \lneq_{c}A_{m-1}\lneq_{c}A_{m}=G\\
 \qquad \text{ and }\\
\mathcal{B}:\quad 1\lneq_{c} \dots \quad \lneq_{c}B_{n-1}\lneq_{c}B_{n}=G,
  \end{align*}
we define the partial order $\preceq $ by
\begin{multline*}
\mathcal{A} \preceq \mathcal{B} \Leftrightarrow \quad
\{A_{i}: 1\leq i \leq m\} \subseteq \{B_{j}: 1\leq j\leq n\}.
\end{multline*}
Next, any chain of finite proper chains
\begin{equation*}
 \mathcal{C}_{1}\preceq \mathcal{C}_{2}\preceq \dots,
\end{equation*}
must terminate under the given hypotheses (i.e., owing to the noetherian and the weak descending chain conditions on $G$). Note that the trivial proper chain $1\lneq_{c} G$ makes the set of all finite chains non-empty. By Zorn's Lemma, we obtain a \textit{maximal} finite proper chain; call it $\mathcal{C}_{max}$.
\begin{align*}
 \mathcal{C}_{max}:\quad 1\lneq_{c} G_{1}\dots \quad \lneq_{c}G_{l-1}\lneq_{c}G_{l}=G.
\end{align*}
Next, we must have $l(G_{k})=k$ for all $1\leq k \leq (l-1)$ owing to the maximality of the chain length of $\mathcal{C}_{max}$. In particular, this implies $rk(G_{l-1}) <\infty$ because of Theorem \ref{thm_noe}.\\

We have $[G:G_{l-1}] = \infty$ (since $\mathcal{C}_{max}$ is proper). We may replace $G_{l-1}$ by a closed subgroup $H$ which is maximal among those which contain $G_{l-1}$ and have infinite index in $G$. (This is afforded by the noetherian property of $G$.) Thus, we may supplant $\mathcal{C}_{max}$ by another maximal proper chain
\begin{equation*}
 \mathcal{C}':\quad 1\lneq_{c} G_{1}\dots \quad \lneq_{c}H\lneq_{c}G_{l}=G.
\end{equation*}
The length of $\mathcal{C}'$ equals that of $\mathcal{C}_{max}$; so we get
\begin{equation*}
l(G_{l-1})= l(H)<\infty.
\end{equation*}
Again, Theorem \ref{thm_noe} forces that $rk(H)<\infty$. But now, Lemma \ref{lem_max} is applicable to $H$ in $G$ and implies that $rk(G)<\infty$. This completes the proof.
\end{proof}
\section{Proof of Theorem \ref{thm_non-open}.}
\begin{proof}
\indent

           Let $ G $ be a pro-$ p $ group satisfying the hypotheses of Thm. \ref{thm_non-open}.  In view of Lemma \ref{not_f.g.}, $ G $ must be finitely generated. It follows by \cite{dixon}(see Prop. 1.7, p. 21) that every open subgroup of $G$ is finitely generated. On the other hand, every closed but non-open subgroup is given to have finite rank. The closed subgroups are, therefore, all finitely generated pro-$p$ groups. Hence, $ G $ is noetherian.\\

           Thus, it suffices to prove that $G$ satisfies the weak descending chain condition in view of Theorem \ref{thm_art}. By \cite{dixon} (see Prop. 1.2(i), p. 15), every closed subgroup of $G$ having a finite index is also open. Hence, no proper descending chain can contain an open subgroup. Therefore, every subgroup of $G$ which preceeds it in some proper chain is closed but non-open (and hence has a finite rank).\\

           But all finite rank closed subgroups of $G$ satisfy the weak descending chain condition (as shown in Theorem \ref{thm_art}). Hence, $G$ satisfies the weak descending chain condition and Theorem \ref{thm_art} applies as claimed. Consequently, we get $rk(G)<\infty$. The converse implication follows trivially; so the proof is complete.
           \end{proof}
\section{Proof of Corollary \ref{cor_final}.}
\begin{proof}
 \indent

    Noetherian pro-$p$ groups of finite rank are virtually $p$-adic analytic and hence countably based. We prove the converse implication now.\\

    So, assume that $G$ is noetherian and countably based. Define the following set consisting of closed subgroup of $G$.
    \begin{equation*}
     \mathcal{H}:= \big \lbrace   H\lneq_{c} G : \quad [G:H]=\infty \text{ and } rk(H)<\infty.  \big \rbrace
    \end{equation*}
    As $G$ is noetherian, $\mathcal{H}$ contains a maximal closed subgroup $H_{max}$ of finite rank. Next, consider
    \begin{equation*}
     H_{\infty}:= \bigcap_{H_{max}\lneq_{c}K\leq_{c}G} K.
    \end{equation*}
    If $rk(G)=\infty,$ then each subgroup $K$ in the above intersection will have infinite rank (owing to the maximality of $H_{max}$). Moreover, each closed but non-open subgroup of $H_{\infty}$ must have finite rank. We observe that $H_{\infty}$ is closed in $G$; so $H_{\infty}$ is noetherian as well as countably based. Thus, Thm. \ref{thm_non-open} is applicable and shows that $rk(H_{\infty})<\infty$. In fact, since $H_{max}\leq_{c}H_{\infty}$, we have $H_{max}=H_{\infty}$ by the maximality of $H_{max}$ among the subgroups of finite rank.\\

    But now, Lemma \ref{lem_max} applied to $G$ yields $rk(G)<\infty$ as well. This contradiction completes the proof.
\end{proof}
\begin{acknowledgement}
    The author would like to thank Chennai Mathematical Institute for support by a post-doctoral fellowship.
\end{acknowledgement}
\bibliographystyle{plain}
\bibliography{pro-p}
\addcontentsline{toc}{chapter}{Bibliography}
\end{document}